\documentclass[12pt, reqno]{amsart}
\usepackage{paralist}

\usepackage{typearea}
\usepackage{geometry}
\usepackage{ulem}
\usepackage{amsmath}
\usepackage{amssymb}
\usepackage{latexsym}
\usepackage{enumerate}
\usepackage{amsthm}
\usepackage[all]{xy}
\usepackage{hhline}
\usepackage{epsf} %fuer bildchen
\usepackage{cite}

\newtheorem{theorem}{{\sc Theorem}}[section]

\newtheorem{prop}[theorem]{{\sc Proposition}}

\theoremstyle{remark}
\newtheorem{remark}[theorem]{{\sc Remark}}

\def\b{\mathbb }

\def\rr{\b R}

\def\nn{\b N}
\def\cc{\b C}
\def\torus{\b T}
\def\zz{\b Z}

\def\erw{\b E}

\def\ll{\b L}

\def\phi{\varphi }

\def\calb{{\mathcal B}}

\def\call{{\mathcal L}}

\def\calw{{\mathcal W}}

\def\on{\operatorname}

\begin{document}

\title[A quantitative central limit theorem for random matrices]{A quantitative central limit theorem for linear statistics of random matrix eigenvalues}
\author{Christian D\"obler}

\author{Michael Stolz}\thanks{Ruhr-Universit\"at Bochum, Fakult\"at f\"ur Mathematik, D-44780 Bochum, Germany.\\
email: christian.doebler@ruhr-uni-bochum.de and michael.stolz@ruhr-uni-bochum.de.\\ Both authors have been supported by Deutsche Forschungsgemeinschaft via SFB-TR 12.
}

\begin{abstract}
It is known that the fluctuations of suitable linear statistics of Haar distributed elements of the compact classical groups satisfy a central limit theorem. We show that 
if the corresponding test functions are sufficiently smooth, a rate of convergence of order almost $1/n$ can be obtained
using a quantitative multivariate CLT for traces of powers that was recently proven using Stein's method of exchangeable pairs.

\end{abstract}

\maketitle

\section{Introduction}

For $n \in \nn$ let $M_n$ denote a random $n \times n$ matrix, distributed according to Haar measure on one of the compact classical groups, i.e.\ 
the unitary, orthogonal, and (if $n$ is even) unitary symplectic group. Its eigenvalues $\lambda_{n1}, \ldots, \lambda_{nn}$
lie on the unit circle line $\torus$ of the complex plane $\cc$. Write $L_n = \ll_n (M_n)$ for the empirical measure $\frac{1}{n} \sum_{j=1}^n \delta_{\lambda_{nj}}$ of the $\lambda_{nj}$.
It is well known (see \cite{MeckesMeckes12+} for a recent quantitative version) that, as $n \to \infty$, $L_n$ tends a.s.\ weakly to the uniform distribution on $\torus$. Furthermore, for suitable test functions $f: \torus \to \rr$,
the fluctuations $n(L_n(f) - \erw L_n(f))$ have been proven to tend to a Gaussian limit as $n \to \infty$ (see, e.g., \cite{DiaconisEvansTAMS, SoshnikovLocal, Wieand02}). 
In fact, these results are part of a broader interest in fluctuations of linear spectral statistics of various random matrix ensembles, both from the physics 
(see the seminal paper \cite{CostinLebowitz})
and mathematics points of view (see, e.g., \cite{ChattFluct} and the references therein).\\

For Haar distributed matrices from the compact classical groups, and in the special case that $f$ is a trigonometric polynomial, 
Johansson has obtained in \cite{JohannssonAnnMath97} an exponential rate of convergence to the Gaussian limit, 
using sophisticated analytic tools related to Szeg\"o's strong limit theorem for Toeplitz determinants. In this note we are concerned with the speed of convergence 
in the case of test functions whose Fourier expansion does not necessarily terminate. 
In this more general setting we have been unable to find quantitative convergence results in the literature. 
We point out that for sufficiently smooth test functions a rate close to $\on{O}(1/n)$
is
an easy consequence of a recent quantitative multivariate CLT for a vector of traces of powers of a Haar distributed element of the compact classical 
groups, where, crucially, the length of the vector may grow with the matrix size $n$.
This CLT was proven by the present authors in \cite{DoeSt11}, using the ``exchangeable pairs'' version of Stein's method (see \cite{CGSbook} for background) and building 
upon a construction of an exchangeable pair that had already
been successfully used by Fulman \cite{Ful10} in the univariate case. Note that a different way of using Stein's method to study the linear eigenvalue statistics of different random matrix ensembles
was devised by Chatterjee in \cite{ChattFluct}.\\

In what follows we concentrate on the case of random unitary matrices. The orthogonal and symplectic cases can be treated along the same lines,
leading to the same rates of convergence towards a (this time not necessarily centered) Gaussian limit. In Section \ref{setup} we recall some 
Fourier analysis as well as the crucial quantitative multivariate CLT for traces of powers. In Section \ref{ratessmoothcase} we state and prove our main result on 
linear statistics.

\section{Set-up and background}
\label{setup}
Let $f \in \on{L}^1(\torus)$ be real valued. 
We will view $f$ as a $2\pi$-periodic function on $\rr$ by tacitly identifying $f(e^{ix})$ with $f(x)$. For $j \in \zz$ the $j$-th Fourier coefficient of $f$
is defined as
$$ \hat{f}_j = \hat{f}(j) = \frac{1}{2\pi} \int_0^{2\pi} f(x) e^{-ijx}dx.$$
Note that since $f$ is real valued, $\hat{f}(-j) = \overline{\hat{f}(j)}.$
It is a well known fact about Fourier coefficients that smoothness of the function implies quantitive information on the decay of the coefficients:

\begin{prop}
\label{coeffbound}
If $f \in \on{C}^k(\torus)$, then for all $0 \neq j \in \zz$ there holds
$$ |\hat{f}(j)| \le  \frac{\|f^{(k)}\|_1}{ |j|^k}.$$
\end{prop}
\begin{proof} See, e.g., \cite[I.4.4]{Katznelson}. \end{proof}

Slightly generalizing this, we will 
consider functions $f \in \on{L}^1(\torus)$ with the property that there exist $\kappa > 1$ and $C_{\kappa} > 0$ such that for all $0 \neq j \in \zz$ there holds
\begin{equation}
\label{kappa}  |\hat{f}(j)| \le  \frac{C_{\kappa}}{ |j|^{\kappa}}.
\end{equation}
Proposition \ref{coeffbound} then says that for $f \in \on{C}^k(\torus)$ $(k \ge 2)$ condition \eqref{kappa} holds with $\kappa = k$, and it follows from the results in \cite[Sec.\ 3.2.2]{GraClassical} that \eqref{kappa} holds with $\kappa = 2$ if $f$ has a Lipschitz first derivative. Note that since $|e^{ijx}| = 1$ for all $j$ and $x$, our assumption $\kappa > 1$ implies that the Fourier expansion of $f$ will converge normally, hence (by compactness of the torus) uniformly  
to $f$. So we have that

$$ \sum_{j \in \zz} \hat{f}(j) e^{ijx} = \hat{f}(0) + \sum_{j=1}^{\infty}\left( \hat{f}(j) e^{ijx} 
+ \overline{\hat{f}(j)\ e^{ijx}}\right).$$ 

Let $M_n$ be a Haar distributed element of $\on{U}_n$ and write
$$ L_n = \ll_n(M_n) = \frac{1}{n} \sum_{j = 1}^n \delta_{\lambda_{nj}},$$ where $\lambda_{n1}, \ldots, \lambda_{nn} \in \torus$ 
are
the eigenvalues (with multiplicities) of $M_n$. Observe that by Fubini's theorem and the left and right invariance of Haar measure one has
\begin{align*}
 \erw L_n(f) &= \frac{1}{n} \erw[f(\lambda_{n1}) + \ldots + f(\lambda_{nn})]\\
&= \frac{1}{n} \int_{\torus} \erw[f(t\lambda_{n1}) + \ldots + f(t\lambda_{nn})] dt \\
&= \frac{1}{n} \erw \int_{\torus} [f(t\lambda_{n1}) + \ldots + f(t\lambda_{nn})] dt \\
&=  \frac{1}{n} \erw\left[ n \int_{\torus} f(t) dt \right] = \int_{\torus} f(t) dt = \hat{f}(0).
\end{align*}
In view of this, it follows from the above that the $n$-scaled fluctuation of $L_n(f)$ has a pointwise expression
\begin{align} \label{mainexpansion} n(L_n(f) - \erw(L_n(f))) &= f(\lambda_{n1}) + \ldots + f(\lambda_{nn}) - n \hat{f}(0)\nonumber \\
&= \sum_{j=1}^{\infty} \hat{f}(j) \on{Tr}(M_n^j) + \sum_{j=1}^{\infty} \overline{\hat{f}(j) \on{Tr}(M_n^j)}.
\end{align}

The following quantitative CLT for vectors of traces of powers of Haar unitaries, proven in \cite{DoeSt11}, will make it possible to control finite sections of
$n$-dependent length of the expansion  \eqref{mainexpansion}:
Let $M = M_n$ be distributed according to Haar measure on $\on{U}_n$. For $d \in \nn,\
r = 1, \ldots, d$, consider the $r$-dimensional complex random vector
$$ W := W(d, r, n) := (\on{Tr}(M^{d-r+1}), \on{Tr}(M^{d-r+2}), \ldots, \on{Tr}(M^{d})).$$ 
Let $Z := (Z_{d-r+1} \ldots, Z_d)$ denote an
$r$-dimensional complex standard normal random vector, i.e., 
there are iid 
real random variables $X_{d-r+1}, \ldots, X_d,$ $Y_{d-r+1}, \ldots, Y_d$ with distribution $\on{N}(0, 1/2)$ such that $Z_j = X_j + i Y_j$ for $j = d-r+1, \ldots, d$. 
We take $\Sigma$ to denote the diagonal matrix $\on{diag}(d-r+1, d-r+2, \ldots, d)$ and write 
$Z_{\Sigma} := \Sigma^{1/2} Z.$ \\

Recall that the Wasserstein distance for probability distributions $\mu$, $\nu$ on $(\cc^d,\calb(\cc^d))$ is defined by 

$$d_\calw(\mu,\nu):=\sup\left\{|\int h d\mu-\int h d\nu|\,:\, h:\cc^d\rightarrow\rr\text{ and } \|h\|_{\on{Lip}}\leq 1\right\},$$
$\|h\|_{\on{Lip}}$ denoting the minimum Lipschitz constant
$$ \|h\|_{\on{Lip}} :=\sup_{x\not=y}\frac{|h(x)-h(y)|}{\|x-y\|_2}\in[0,\infty]$$ of $h$. If $\mu$ is the law of $X$ and $\nu$ is the law of $Y$, we write $d_{\calw}(X, Y) = d_{\calw}(\mu, \nu).$\\

We are now in a position to state the crucial quantitative CLT for vectors of traces of powers \cite[Thm.\ 1.1]{DoeSt11}.

\begin{prop}
\label{trofpow} If $n \ge 2d$, the Wasserstein distance between $W$ and $Z_{\Sigma}$ is
\begin{equation} 
\label{ordnungsformel}
d_\calw(W,Z_\Sigma)=\on{O}\left(\frac{\max\left\{\frac{r^{7/2}}{(d-r+1)^{3/2}},\,(d-r)^{3/2}\sqrt{r}\right\}}{n}\right)\,.\end{equation}
In particular, for $r = d$ we have
\begin{equation}\label{trofpow-speziell}
 d_{\calw}(W, Z_{\Sigma}) = \on{O}\left(\frac{d^{7/2}}{n}\right).\end{equation}
\end{prop} 

We will also need the following orthogonality relations for traces of powers:

\begin{prop}
\label{orth}
Let $n, i, j \in \nn$. Then
$$\erw(\on{Tr}(M_n^i) \on{Tr}(M_n^j)) = 0$$ and
$$ \erw(\on{Tr}(M_n^i) \overline{\on{Tr}(M_n^j)}) = \delta_{ij} (j \wedge n).$$
\end{prop}
\begin{proof} The first equality follows from the unitary invariance of Haar measure, see also \cite[Thm.\ 3.29]{StolzDSM} For the second equality see 
\cite[Thm.2.1(b)]{DiaconisEvansTAMS}.
\end{proof}

\section{Rates of convergence for linear statistics}
\label{ratessmoothcase}

The following theorem is the main result of this note. 
\begin{theorem}
Let $M_n$ be a Haar distributed unitary $n \times n$ matrix, and let $f \in \on{L}^1(\torus)$ be real valued and satisfy
condition \eqref{kappa} above with $\kappa > 1$. Then, as $n \to \infty$, the fluctuation
$n(L_n(f) - \erw L_n(f))$ converges in distribution to a centered Gaussian random variable with variance
$$ \sigma^2 := \sum_{j = -\infty}^{\infty} |\hat{f}_j|^2 |j| = 2 \sum_{j=1}^{\infty} |\hat{f}_j|^2 j < \infty.$$ If $\call_n$ denotes the law of  $n(L_n(f) - \erw L_n(f))$,
then there exists $C = C(f, \kappa) > 0$ such that for any $n \in \nn$ the Wasserstein distance between $\call_n$ and $\on{N}(0, \sigma^2)$ can be bounded as follows:

$$ d_\calw(\call_n, \on{N}(0, \sigma^2)) \le C\ n^{-\left(\frac{k-1}{k + 5/2}\right)}.$$

\vspace{1em}
In particular, if $f$ is in $\on{C}^{\infty}(\torus)$, then $d_\calw(\call_n, \on{N}(0, \sigma^2))$ is $\on{O}\left( \frac{1}{n^{1-\epsilon}}\right)$ for all $\epsilon > 0$.
\end{theorem}

\begin{remark}
In \cite{DiaconisEvansTAMS} the corresponding weak convergence result without a bound
on the speed of convergence is proven under the weaker assumption that 
$\sigma^2$ be finite, which amounts to requiring that $f$ be an element
of the Sobolev space $H^{1/2}(\torus)$. Since on the one hand, elements of 
$H^{1/2}(\torus)$ need not be continuous \cite[I.8.11]{Katznelson}, and on the other hand, our method of proof requires the assumption that $\kappa > 1$, which implies that $f$ must be the uniform limit of a sequence of continuous functions, the present result applies to a strictly smaller class of test functions.
\end{remark}

\begin{proof}
It follows from \eqref{kappa} that $\sigma^2 < \infty.$ 
Set $$S_n :=  \sum_{j=1}^{\infty} \hat{f}(j) \on{Tr}(M_n^j) + \sum_{j=1}^{\infty} \overline{\hat{f}(j) \on{Tr}(M_n^j)},$$
and write, for $d = d(n) \le \frac{n}{2}$ to be chosen later,
$$ S^{(1)}_n :=  \sum_{j=1}^{d} \hat{f}(j) \on{Tr}(M_n^j) + \sum_{j=1}^{d} \overline{\hat{f}(j) \on{Tr}(M_n^j)},$$
and 
\begin{equation} \label{S2} S_n^{(2)} := S_n - S_n^{(1)}.\end{equation} Then, by Proposition \ref{coeffbound} and
dominated convergence, 
\begin{align*}
\erw|S_n^{(2)}|^2 &= \sum_{i, j = d+1}^{\infty} \erw\left( \hat{f}_i \on{Tr}(M_n^i) + \overline{\hat{f}_i}\ \overline{\on{Tr}(M_n^i)}\right)
\overline{\left( \hat{f}_j \on{Tr}(M_n^j) + \overline{\hat{f}_j}\ \overline{\on{Tr}(M_n^j)}\right)}\\
&= \sum_{i, j = d+1}^{\infty} \hat{f}_i \overline{\hat{f}_j}\ \erw\left( \on{Tr}(M_n^i) \overline{\on{Tr}(M_n^j)}\right) + \hat{f}_i \hat{f}_j\ \erw\left( \on{Tr}(M_n^i) \on{Tr}(M_n^j)\right) \\
&\quad\quad +  \overline{\hat{f}_i \hat{f}_j}\ \overline{ \erw\left( \on{Tr}(M_n^i) \on{Tr}(M_n^j)\right)} +  \overline{\hat{f}_i} \hat{f}_j\ \erw\left( \overline{\on{Tr}(M_n^i)} \on{Tr}(M_n^j)\right)\\
&= 2 \sum_{j=d+1}^{\infty} |\hat{f}_j|^2 (j \wedge n),
\end{align*} 
where we have used Proposition \ref{orth} in the last step. Consequently, Proposition \ref{coeffbound} implies that 
\begin{align}
\erw|S_n^{(2)}|^2 &\le 2 \sum_{j = d+1}^{\infty} j |\hat{f}_j|^2 \le 2 \| f^{(k)}\|_1^2 \sum_{j = d+1}^{\infty} \frac{1}{ j^{2k - 1}}\nonumber\\
&\le 2 \| f^{(k)}\|_1^2 \int_d^{\infty} \frac{1}{x^{2k - 1}} dx = 2 \| f^{(k)}\|_1^2 \frac{1}{(2k - 2) d^{(2k-2)}}.\label{boundS2}
\end{align}

For $z \in \cc^d$ define $$\phi(z) = \phi_n(z) = \sum_{j=1}^d \hat{f}_j z_j + \overline{ \hat{f}_j} \overline{z_j}.$$
We write $\underline{f} = (\hat{f}_1, \ldots, \hat{f}_d)$. Then, for any $z, w \in \cc^d$, Cauchy-Schwarz implies
$$|\phi(z) - \phi(w)| = \left| \sum_{j=1}^d \hat{f}_j (z_j - w_j) + \overline{\hat{f}_j}(\overline{z_j} - \overline{w_j})
\right| \le 2 \| \underline{f}\|_2 \| z-w\|_2.$$ We have thus obtained that $$\|\phi\|_{\on{Lip}} \le
2 \| \underline{f}\|_2.$$ Note that, by Proposition \ref{coeffbound}, 
$$ \| \underline{f}\|^2_2\ = \sum_{j=1}^{d(n)} |\hat{f}_j|^2 \le \|f^{(k)}\|_1^2 \sum_{j = 1}^{\infty} \frac{1}{j^{2k}} < \infty$$ is bounded by a constant independent of $n$ and of the 
specific form of $f$.\\

Now let $\zeta$ be a real standard normal random variable, and $Z, \Sigma, Z_{\Sigma}$ as above in Section \ref{setup}.
Let $g: \rr \to \rr$ be Lipschitz with $\|g\|_{\on{Lip}} \le 1$. Then
\begin{align}
&|\erw(g(S_n)) - \erw(g(\sigma\zeta))| \nonumber\\
&\le\ |\erw(g(S_n)) - \erw(g(S_n^{(1)}))| + |\erw(g(S_n^{(1)})) - \erw(g(\phi(Z_{\Sigma})))| + |\erw(g(\phi(Z_{\Sigma}))) - \erw(g(\sigma \zeta)))|
\label{EinsZweiDrei}
\end{align}
In view of \eqref{S2} and \eqref{boundS2}, the first summand in \eqref{EinsZweiDrei} can be bounded by
$$ \|g\|_{\on{Lip}}\ \erw|S_n^{(2)}| \le \|g\|_{\on{Lip}} \sqrt{ \erw|S_n^{(2)}|^2 }
\le \|g\|_{\on{Lip}}\ \|f^{(k)}\|_1 \sqrt{ \frac{1}{(2k - 2)\ d^{2k-2}}}.
$$
Writing $W_n := (\on{Tr}(M_n), \ldots, \on{Tr}(M_n^d))$, by Proposition \ref{trofpow} the second summand in \eqref{EinsZweiDrei} has an upper bound
\begin{align*}
|\erw(g(S_n^{(1)})) - \erw(g(\phi(Z_{\Sigma})))| &= |\erw(g(\phi(W_n))) - \erw(g(\phi(Z_{\Sigma})))|\\
&= \|\phi\|_{\on{Lip}} \left| \erw \left( \frac{1}{\|\phi\|_{\on{Lip}}} (g \circ \phi)(W_n)\right) - 
\erw \left( \frac{1}{\|\phi\|_{\on{Lip}}} (g \circ \phi)(Z_{\Sigma})\right)\right|\\
&\le \|\phi\|_{\on{Lip}}\ d_{\calw}(W_n, Z_{\Sigma}) \le 2 \| \underline{f}\|_2\ C\ \frac{d^{7/2}}{n},
\end{align*}
where $C$ is the constant from Proposition \ref{trofpow}. \\

To bound the third summand in \eqref{EinsZweiDrei}, we will first study the distribution of $\phi(Z_{\Sigma})$.
Recall that $Z_{\Sigma}$ is the random vector $(Z_1, \sqrt{2} Z_2, \ldots, \sqrt{d}Z_d)$, where the $Z_j\ (j = 1, \ldots,
d)$ are iid standard complex normal random variables, i.e., 
there are iid 
real random variables $X_{1}, \ldots, X_d,$ $Y_{1}, \ldots, Y_d$ with distribution $\on{N}(0, 1/2)$ such that $Z_j = X_j + i Y_j$ for $j = 1, \ldots, d$. 
Then 
\begin{align*}
\phi(Z_{\Sigma}) &= \sum_{j=1}^d \hat{f}_j \sqrt{j} Z_j + \sum_{j=1}^d \overline{\hat{f}_j} \sqrt{j}\ \overline{Z_j}\\
&= 2 \sum_{j=1}^d \sqrt{j}\ \on{Re}( \hat{f}_j  Z_j) \\
&= 2 \sum_{j=1}^d \sqrt{j}\ \left(  \on{Re}( \hat{f}_j) X_j - \on{Im}( \hat{f}_j) Y_j\right) 
\end{align*}
$\phi(Z_{\Sigma})$ is thus a centered real normal random variable with variance
$$ 4 \sum_{j=1}^d \frac{j}{2} (\on{Re}(\hat{f}_j)^2 + \on{Im}(\hat{f}_j)^2) = 2 \sum_{j=1}^{d(n)} j |\hat{f}_j|^2
=: \sigma_n^2.$$
So we obtain
\begin{align*}
|\erw (g(\phi(Z_{\Sigma}))) - \erw(g(\sigma \zeta))| &= |\erw (g(\sigma_n \zeta)) - \erw(g(\sigma \zeta))|\\
&\le\ \|g\|_{\on{Lip}}\ \erw|\sigma_n \zeta - \sigma \zeta|\\ &\le \|g\|_{\on{Lip}}\ |\sigma_n - \sigma|\ \erw|\zeta|\\ &\le  \|g\|_{\on{Lip}}\ |\sigma_n - \sigma|.
\end{align*}
Now,
\begin{align*}
|\sigma_n^2 - \sigma^2| &= 2 \sum_{j = d+1}^{\infty} j |\hat{f}_j|^2 \le 2 \|f^{(k)}\|_1^2\ \sum_{j = d+1}^{\infty} \frac{1}{j^{2k - 1}}\\ &\le 
2 \|f^{(k)}\|_1^2\ \int_d^{\infty} \frac{1}{x^{2k-1}} dx\\
&= 2 \|f^{(k)}\|_1^2\ \frac{1}{(2k - 2)\ d^{2k - 2}}.\end{align*}
Since $|\sigma_n^2 - \sigma^2| = |(\sigma_n - \sigma) (\sigma_n + \sigma)|$, we have that
$$ |\sigma_n - \sigma| = \frac{|\sigma_n^2 - \sigma^2|}{ \sigma_n + \sigma} \le \frac{| \sigma_n^2 - \sigma^2|}{\sigma},$$
and the third term in \eqref{EinsZweiDrei} may thus be bounded by
\begin{equation}
\label{BoundDrei}
2 \|g\|_{\on{Lip}}\  \|f^{(k)}\|_1^2\  \frac{1}{ \sigma (2k-2)\ d^{2k-2}}.
\end{equation}

Comparing the bounds that we have obtained for the individual summands in \eqref{EinsZweiDrei}, we obtain that
\begin{align}
\label{ZusammenschauSchranken}
|\erw(g(S)) - \erw(g(\sigma \zeta))| &\le \|g\|_{\on{Lip}}\ C_{f, k}\ \max\left\{ \frac{1}{d^{k-1}},\ \frac{d^{7/2}}{n}\right\}\nonumber\\
&\le  \|g\|_{\on{Lip}}\ C_{f, k}\ \left( \frac{1}{d^{k-1}} + \frac{d^{7/2}}{n}\right),
\end{align}
where $C_{f, k}$ depends only on $f$ and $k$.\\

One verifies that, on $]0, \infty[$, the function $f(x) :=  \frac{1}{x^{k-1}} + \frac{x^{7/2}}{n}$ has a unique global minimum at
$$ x_0 := \left(\frac{2(k-1)}{7}\right)^{\frac{2}{5+2k}} n^{\frac{2}{5+2k}}.$$
Furthermore, $f$ is strictly falling on $]0, x_0[$ and strictly growing on $]x_0, \infty[$. Setting $d_1 := \lfloor x_0 \rfloor$ and $d_2 := d_1 + 1$, we see
that the minimizer of $f$ on the positive integers is in $\{ d_1, d_2\}$. \\

There are $0 < \alpha < \beta$, depending only on $k$, such that for all $n \in \nn$ one has 
$$ \alpha  n^{\frac{2}{5+2k}} \le d_1 < d_2 \le \beta  n^{\frac{2}{5+2k}}.$$ Hence,
\begin{align*}
 f(x_0) &\le \frac{1}{d_1^{k-1}} + \frac{d_2^{7/2}}{n} \le \frac{1}{\alpha^{k-1} n^{\frac{2k - 2}{2k + 5}}} + \beta^{7/2} n^{\left( \frac{7}{5+2k} - 1\right)}\\
 &= \frac{1}{\alpha^{k-1} n^{\frac{2k - 2}{2k + 5}}} + \beta^{7/2} n^{- \frac{2k - 2}{2k + 5}} = \gamma n^{- \frac{2k - 2}{2k + 5}},
\end{align*}
where $\gamma := \frac{1}{\alpha^{k-1}} + \beta^{7/2}.$ Comparing with \eqref{ZusammenschauSchranken} yields the theorem.
\end{proof}

\end{document}